\documentclass[12pt,a4paper]{amsart}

\usepackage{amscd,amssymb,amsthm}
\usepackage{graphicx}

\def\ds{\displaystyle}
\numberwithin{equation}{section}

\theoremstyle{plain}
\newtheorem{theorem}{Theorem}[section]
\newtheorem{corollary}[theorem]{Corollary}
\newtheorem{lemma}[theorem]{Lemma}
\newtheorem{proposition}[theorem]{Proposition}

\theoremstyle{definition}

\theoremstyle{remark}

\numberwithin{equation}{section}

\numberwithin{table}{section}

\numberwithin{figure}{section}

\def\ds{\displaystyle}

\newcommand{\bea}{\begin{eqnarray*}}
\newcommand{\eea}{\end{eqnarray*}}
\newcommand{\bean}{\begin{eqnarray}}
\newcommand{\eean}{\end{eqnarray}}

\begin{document}

\title{On the zeros of the big $q$-Bessel functions and applications}
\author{Fethi BOUZEFFOUR}
\address{Department of mathematics, King Saudi University, College of Sciences\\ P. O Box 2455 Riyadh 11451, Saudi Arabia.
} \email{fbouzaffour@ksu.edu.sa}

\author{Hanen Ben Mansour }
\address{Department of mathematics, University of Carthage, Faculty of Sciences of Bizerte\\ $7021$ Tunisia}
\email{benmansourhanen52@yahoo.com} 

\begin{abstract}
This paper deals with the study of the zeros of the big $q$-Bessel
functions. In particular, we prove a new orthogonality relations for
this functions similar to the one for the classical Bessel
functions. Also we give some applications related to the sampling
theory.
\end{abstract}
\maketitle
\section{introduction}
The classical Bessel functions $ J_{\alpha}(x)$ which are defined by
\cite{wat}
\begin{eqnarray*}
\\J_{\alpha}(x):
=\frac{(x/2)^{\alpha}}{\Gamma(\alpha+1)}\  _{0}F_{1}\left(
  \begin{array}{cc}
    -\\
    \alpha+1\\
  \end{array}
 ;-\frac{x^{2}}{4} \right),
\end{eqnarray*}  satisfy the orthogonality relations
\begin{equation}
\int_0^1J_{\alpha}(j_{\alpha  n}x)J_{\alpha}(j_{\alpha  m}x)xdx=\frac{\delta_{n m}}{2}J^2_{\alpha+1}(j_{\alpha  n}x),
\end{equation}
where $\{j_{\alpha  n }\}_{n\in \mathbb{N}}$ are the zeros of $J_{\alpha}(x)$.\\ Moreover, a function  $f \in L^2((0,1);\,xdx),$  can be represented as the Fourier-Bessel series
 \begin{equation}
 f(x)=\sum_{n=0}^{\infty}c_nJ_{\alpha}(j_{\alpha n}x),
\end{equation}
where
\begin{equation}
c_n=\frac{2}{J^2_{\alpha+1}(j_{\alpha  n}x)}\int_0^1 f(x)J_{\alpha}(j_{\alpha  n}x)xdx.
\end{equation}
In the literatures there are many basic extension of the Bessel
functions $J_{\alpha}(x)$, the oldest one have been introduced by
Jackson in $1903-1905$ and rewritten in modern notation by Ismail
\cite{ism}. Other $q$-analogues can be obtained as formal limit of
the three $q$-analogues of Jacobi polynomials, i.e., of little
$q$-Jacobi polynomials, big $q$-Jacobi polynomials and Askey-Wilson
polynomials for this reason we propose to speak about little
$q$-Bessel functions, big $q$-Bessel functions and AW type
$q$-Bessel functions for the corresponding limit cases.\\Recently,
Koelink and Swartouw establish an orthogonality relations for the
little $q$-Bessel (see \cite{Koelink}). Another orthogonality
relations for Askey-Wilson functions founded by Bustoz and Suslov
(see, \cite{Bus}). In this paper we shall discuss a new
orthogonality relations for the big $q$-Bessel functions \cite{c}
\begin{equation} J_{\alpha}(x,\lambda;q^2)=\ _{1}\phi_1\left(\left. \begin{matrix}-1/x^2\\
q^{2\alpha+2}\end{matrix}\right\vert q^2;\lambda^2 x^2 q^{2\alpha+2}\right) \label{11}.
\end{equation}
In section 2, we define the big $q$-Bessel function, we give some
recurrence relations and we prove that the big $q$-Bessel is an
eigenfunction of a $q$-difference equation of second order. The
section 3, is devoted to study the zeros of the big $q$-Bessel
functions. In section 4, we show that if \ $\{j_{n}^{\alpha}\}_{n=1}^{\infty}$ \ are
the zeros of \ $J_{\alpha+1}(a,\lambda;q^{2})$, the set of
functions \
$\{J_{\alpha+1}(a,j_{n}^{\alpha};q^{2})\}_{n=1}^{\infty}$  is
a complete orthogonal system in $
L_{q}^{2}((0,1);\frac{t(-t^{2}q^{2};q^{2})_{\infty}}{(-t^{2}q^{2\alpha+4};q^{2})_{\infty}}d_qt)$.
Finally in the last section, we give a version of the sampling
theorem in the points $j_{n}^{\alpha}$.
\section{The big $q$-bessel functions}
Throughout this paper we will fix $q\in]0,1[$. For the convenience
of the reader, we provide a summary of the notations and definitions
used in this paper.\\
Let $a \in \mathbb{C}$ and $n \in\mathbb{N}$, the $q$-shifted
factorials are defined by \cite{gasper}
\begin{equation}
(a;q)_{0}:=1, \ \ \ \ \ \ (a;q)_{n}:=\prod\limits_{i=0}^{n-1}(1-aq^{i}),
\end{equation}
and
\begin{equation}
(a_{1},a_{2},...,a_{k};q)_{n}:=\prod\limits_{j=1}^{k}(a_{j};q)_{n}.
\end{equation}
We also denote
  $$ (a;q)_{\infty}= \lim\limits_{n\rightarrow \infty}(a;q)_{n}.
$$
The basic hypergeometric series $_{r}\phi_{s}$ is defined by
\cite{gasper}
\begin{equation}
   \ _{r}\phi_{s} \left(\left. \begin{matrix} \ a_{1}, a_{2}, ..., a_{r}\\
  \ b_1,b_{2},...,b_s\end{matrix} \right\vert q,z\right):=\ds\sum_{k=0}^{+\infty}\ds\frac{\left(a_{1}, a_{2},
. . . , a_{r} \ ; q \right)_{k}}{\left(q,b_{1}, b_{2}, . . . , b_{s} \
; q \right)_{k}}
(-1)^{k(1+s-r)}q^{(1+s-r)\binom{k}{2}}z^{k}.\end{equation}
whenever the series converges.\\
The $q$-derivative $D_qf(x)$ of a function $f$ is defined by
\begin{equation}
(D_{q}f)(x)=\frac{f(x)-f(qx)}{(1-q)x},\ \ \ \ \ \ \ \text{if} \ x\neq0,
\end{equation}
and \ $(D_{q}f)(0)=f'(0)$ provided $f'(0)$ exists.\\
The $q$-integral of a function $f$ from $0$ into $a$ is defined by
\begin{equation}
\int_{0}^{a}f(t) d_{q}t=(1-q)a\sum\limits_{n=0}^{\infty}f(aq^{n})q^{n}.
\end{equation}
The $q$-integration by parts formula is given by
\begin{equation}
\int_{a}^{b}D_{q^{-1}}[g(x)]f(x)d_{q}x=q[f(x) g(q^{-1}x)]_{a}^{b}-q\int_{a}^{b}D_{q}(f(x))g(x) d_{q}x.\label{int}
\end{equation}
The big $q$-Bessel functions are defined by \cite{c}
\begin{eqnarray} J_{\alpha}(x,\lambda;q^2)&=&\ _{1}\phi_1\left(\left. \begin{matrix}-1/x^2\\
q^{2\alpha+2}\end{matrix}\right\vert q^2;\lambda^2 x^2 q^{2\alpha+2}\right), \label{11}\\
&=&\sum\limits_{k=0}^{\infty}\frac{(-1)^{k}q^{2(^{k}_{2})+2k(\alpha+1)}}{(q^{2},q^{2\alpha+2};q^{2})_{k}}(-\frac{1}{x^{2}};q^{2})_{k}(\lambda x)^{2k}.\nonumber
\end{eqnarray}
For $\alpha>-1$, the functions  $J_{\alpha}(x,\lambda;q^{2})$ are
analytic in  $\mathbb{C}$ in their variables  $x$ and $\lambda$
and satisfying
\begin{equation}
 \lim_{q\rightarrow1}J_{\alpha}(\frac{x}{1-q^{2}},(1-q^{2})^{2}\lambda;q^{2})=\ _{0}F_{1}\left(
  \begin{array}{cc}
    -\\
    \alpha+1\\
  \end{array}
 ;-\frac{(2\lambda x)^{2}}{4} \right)
 =j_{\alpha}(2\lambda x)
 \end{equation}
where $j_{\alpha}(.)$  is the normalized Bessel function of order \
$\alpha$ given by
\begin{align*}
 j_{\alpha}(x)&=(x/2)^{-\alpha}\Gamma(\alpha+1)J_{\alpha}(x),\\&=\ _{0}F_{1}\left(
  \begin{array}{cc}
    -\\
    \alpha+1\\
  \end{array}
 ;-\frac{ x^{2}}{4} \right).
 \end{align*}.
\begin{proposition}
The big $q$-Bessel functions satisfy the following recurrence
difference relations
\begin{align}
 &D_{q}\big[J_{\alpha}(x,\lambda;q^{2})\big]=-\frac{\lambda^{2} q^{2\alpha+2}x}{(1-q)(1-q^{2\alpha+2})}J_{\alpha+1}(x,\lambda;q^{2}),\label{rec} \\ \label{recc}
 &D_{q^{-1}}\big[\frac{(-x^{2} q^{2};q^{2})_{\infty}}{(-x^{2}q^{2\alpha+4};q^{2})_{\infty}}J_{\alpha+1}(x,\lambda;q^{2})\big]=-\frac{x(1-q^{2\alpha+2})(-x^{2} q^{2};q^{2})_{\infty}}{(1-q^{-1})(-x^{2}q^{2\alpha+2};q^{2})_{\infty}}J_{\alpha}(x,\lambda;q^{2}).
\end{align}
\end{proposition}
\begin{proof} A simple calculation show that
$$D_{q}\big[(-\frac{1}{x^{2}};q^{2})_{k}x^{2k}\big]=\frac{(1-q^{2k})}{(1-q)}(-\frac{1}{x^{2}};q^{2})_{k-1}x^{2k-1}.$$
Hence,
\begin{eqnarray}
D_{q}\big[J_{\alpha}(x,\lambda ;q^{2})\big]&=&
\sum\limits_{k=1}^{\infty}\frac{(-1)^{k}q^{2(_{2}^{k})+2k(\alpha+1)}}{(q^{2},q^{2\alpha+2};q^{2})_{k}} \lambda^{2k}D_{q}\big[(-\frac{1}{x^{2}};q^{2})_{k}x^{2k}\big]
,\nonumber\\
&=&\frac{1}{(1-q)(1-q^{2\alpha+1})}\sum\limits_{k=1}^{\infty}\frac{(-1)^{k}q^{2(_{2}^{k})+2k(\alpha+1)}}{(q^{2},q^{2\alpha+4};q^{2})_{k-1}}\lambda ^{2k}
(-\frac{1}{x^{2}};q^{2})_{k-1}x^{2k-1} \label{diffrec}.
\end{eqnarray}
Then, we obtain after making the change $k\rightarrow k+1$ in the
second member of \eqref{diffrec}.
\begin{eqnarray*}
D_{q}\big[J_{\alpha}(x,\lambda;q^{2})\big]
&=&-\frac{\lambda^{2} q^{2\alpha+2}x}{(1-q)(1-q^{2\alpha+2})}\sum\limits_{k=0}^{\infty}\frac{(-1)^{k}q^{2(^{k}_{2})+2k(\alpha+2)}}{(q^{2},q^{2\alpha+4};q^{2})_{k}}
(-\frac{1}{x^{2}};q^{2})_{k}(\lambda x)^{2k},\\
&=&-\frac{\lambda^{2} q^{2\alpha+2}x}{(1-q)(1-q^{2\alpha+2})}J_{\alpha+1}(x,\lambda;q^{2}).
\end{eqnarray*}
On the other hand from the following relation
$$ D_{q^{-1}}[\frac{(-x^{2}
q^{2};q^{2})_{\infty}}{(-x^{2}q^{2
\alpha+4};q^{2})_{\infty}}(-\frac{1}{x^{2}};q^{2})_{k}x^{2k}]=\frac{q^{-2k}(q^{2
\alpha+2+2k}-1)}{1-q^{-1}}\frac{(-x^{2}
q^{2};q^{2})_{\infty}}{(-x^{2} q^{2
\alpha+2};q^{2})_{\infty}}(-\frac{1}{x^{2}};q^{2})_{k}x^{2k+1}, $$
we obtain
 \begin{eqnarray*}
&&D_{q^{-1}}\big[\frac{(-x^{2} q^{2};q^{2})_{\infty}}{(-x^{2}q^{2\alpha+4};q^{2})_{\infty}}J_{\alpha+1}(x,\lambda;q^{2})\big]\\
&=&-\frac{x(1-q^{2\alpha+2})(-x^{2} q^{2};q^{2})_{\infty}}{(1-q^{-1})(-x^{2}q^{2\alpha+2};q^{2})_{\infty}}\sum\limits_{k=0}^{\infty}
\frac{(-1)^{k}q^{2(_{2}^{k})+2k(\alpha+1)}}{(q^{2},q^{2\alpha+2};q^{2})_{k}}(-\frac{1}{x^{2}};q^{2})_{k}(\lambda x)^{2k},\\
&=&-\frac{x(1-q^{2\alpha+2})(-x^{2} q^{2};q^{2})_{\infty}}{(1-q^{-1})(-x^{2}q^{2\alpha+2};q^{2})_{\infty}}J_{\alpha}(x,\lambda;q^{2}).
\end{eqnarray*}
\end{proof}
In the classical case the trigonometric functions $\sin(x)$ and $\cos(x)$ are related to the Bessel function $J_{\alpha}(x)$ by
$$\cos(x)=\sqrt{\frac{\pi x}{2}}J_{-\frac{1}{2}}(x),$$
$$\sin(x)=\sqrt{\frac{\pi x}{2}}J_{\frac{1}{2}}(x).$$
Similarly, there are two big $q$-trigonometric functions associated to the big $q$-Bessel function given
\begin{eqnarray*}
\cos(x,\lambda;q^{2})&=&J_{-1/2}(x,\lambda;q^{2}),\\
&=&\sum\limits_{k=0}^{\infty}\frac{(-1)^{k}q^{2(^{k}_{2})+k}}{(q;q)_{2k}}(-\frac{1}{x^{2}};q^{2})_{k}(\lambda x)^{2k},
\end{eqnarray*}
and,
\begin{eqnarray*}
\sin(x,\lambda;q^{2})&=&\frac{1}{1-q}J_{1/2}(x,\lambda;q^{2}),\\
&=&\sum\limits_{k=0}^{\infty}\frac{(-1)^{k}q^{2(^{k}_{2})+3k}}{(q;q)_{2k+1}}(-\frac{1}{x^{2}};q^{2})_{k}(\lambda x)^{2k}.
\end{eqnarray*}
We have
$$D_{q}\big[\cos(x,\lambda;q^{2})\big]=-\frac{\lambda^{2}qx}{1-q}\sin(x,\lambda;q^{2}),$$
$$D_{q^{-1}}\big[\frac{(-x^{2} q^{2};q^{2})_{\infty}}{(-x^{2}q^{3};q^{2})_{\infty}}\sin(x,\lambda;q^{2})\big]=-xq(1-q)^{2}\frac{(-x^{2} q^{2};q^{2})_{\infty}}{(-x^{2}q;q^{2})_{\infty}}\cos(x,\lambda;q^{2}).$$
Let $ L$ be the linear $ q$-difference operator defined by
$$ Lf(x)=\frac{(-x^{2}q^{2\alpha+2};q^{2})_{\infty}}{(-x^{2} q^{2};q^{2})_{\infty}x}D_{q^{-1}}\big[\frac{(-x^{2} q^{2};q^{2})_{\infty}}{(-x^{2}q^{2\alpha+4};q^{2})_{\infty}x}D_{q} f(x)\big]. $$
\begin{theorem}
The big $q$-Bessel function is solution of the $q$-difference equation:
\begin{equation*}
Lf=\frac{-\lambda^{2} q^{2\alpha+3}}{(1-q)^{2}}f.
\end{equation*}
\end{theorem}
\begin{proof}
By \eqref{rec} we have
$$
J_{\alpha+1}(x,\lambda;q^{2})=-\frac{1-q}{\lambda^{2}
q^{2\alpha+2}x}D_{q}\big[J_{\alpha}(x,\lambda;q^{2})\big],$$ and by \eqref{recc} we can be write
$$\frac{(-x^{2}q^{2\alpha+2};q^{2})_{\infty}}{(-x^{2} q^{2};q^{2})_{\infty}x}D_{q^{-1}}\big[\frac{(-x^{2} q^{2};q^{2})_{\infty}}{(-x^{2}q^{2\alpha+4};q^{2})_{\infty}x}D_{q}[J_{\alpha}(x,\lambda;q^{2})]\big]=\frac{-\lambda^{2} q^{2\alpha+3}}{(1-q)^{2}}J_{\alpha}(x,\lambda;q^{2}). $$
\end{proof}
\begin{proposition} The big $q$-Bessel functions satisfy the recurrence relations
\begin{align*}
&i)J_{\alpha+1}(x,\lambda;q^{2})=\frac{(1-q^{2\alpha+2})}{\lambda^{2} q^{2\alpha}(q^{2\alpha+2}x^{2}+1)}\big[(1-q^{2\alpha}-\lambda^{2}q^{2\alpha} x^{2})J_{\alpha}(x,\lambda;q^{2})-(1-q^{2\alpha})J_{\alpha-1}(x,\lambda;q^{2})\big],\label{reccu} \\
&ii)J_{\alpha+1}(x q^{-1},\lambda;q^{2})=\frac{(1-q^{2\alpha+2})}{\lambda^{2} q^{2\alpha}(1+x^{2})}\big[(1-q^{2\alpha})J_{\alpha}(x,\lambda;q^{2})-J_{\alpha-1}(x,\lambda;q^{2})\big].
\end{align*}
\end{proposition}
\begin{proof}$i)$ By \eqref{recc} and \eqref{rec} we get
\begin{equation}
(1+x^{2} q^{2\alpha+2})J_{\alpha}(x q,\lambda;q^{2})-(1+x^{2} q^{2})J_{\alpha}(x,\lambda;q^{2})=-x^{2} q^{2} (1-q^{2\alpha})J_{\alpha-1}(xq,\lambda;q^{2}),\label{r}
\end{equation}
and
\begin{equation*}
J_{\alpha}(x q,\lambda;q^{2})=J_{\alpha}(x ,\lambda;q^{2})+\frac{\lambda^{2} x^{2} q^{2\alpha+2}}{1-q^{2\alpha+2}} J_{\alpha+1}(x,\lambda ;q^{2}).
\end{equation*}
Hence,\begin{align*}
J_{\alpha+1}(x,\lambda;q^{2})=\frac{(1-q^{2\alpha+2})}{\lambda^{2} q^{2\alpha}(q^{2\alpha+2}x^{2}+1)}\big[(1-q^{2\alpha}-\lambda^{2}q^{2\alpha} x^{2})J_{\alpha}(x,\lambda;q^{2})-(1-q^{2\alpha})J_{\alpha-1}(x,\lambda;q^{2})\big].
\end{align*}
$ii)$  Similarly, the equality \eqref{rec} gives us
\begin{equation}
J_{\alpha}(x,\lambda ;q^{2})=J_{\alpha}(x q,\lambda;q^{2})-\frac{\lambda^{2} x^{2} q^{2\alpha+2}}{1-q^{2\alpha+2}}J_{\alpha+1}(x,\lambda ;q^{2}),\label{reccuren}
\end{equation}
then by \eqref{r} and \eqref{reccuren}, we obtain
 \begin{align*}
J_{\alpha+1}(x q^{-1},\lambda;q^{2})=\frac{(1-q^{2\alpha+2})}{\lambda^{2} q^{2\alpha}(1+x^{2})}\big[(1-q^{2\alpha})J_{\alpha}(x,\lambda;q^{2})-J_{\alpha-1}(x,\lambda;q^{2})\big].
\end{align*}
\end{proof}
\section{On the zeros of the big \ $q$-Bessel functions}
By using a similar method as \cite{Koelink}, we prove in this
section that the big $q$-Bessel function has infinity of simple
zeros on the real line and by an explicit evaluation of a
$q$-integral, we established a new orthogonality relations for this
function.
 \begin{proposition}
Let $\alpha >-1$ and $ a>0. $ For every  $\lambda ,
\mu\in\mathbb{C}\setminus\{0\}$, we have
\begin{equation}
(\lambda^{2} - \mu^{2})\int_{0}^{a}\frac{x(-x^{2} q^{2};q^{2})_{\infty}}{(-x^{2}q^{2\alpha +4};q^{2})_{\infty}}J_{\alpha+1}(x,\lambda;q^{2})J_{\alpha+1}(x,\mu;q^{2})d_{q}x
\end{equation}
$$=\frac{(1-q)(1-q^{2\alpha+2})}{q^{2\alpha+2}}\frac{(-a^{2};q^{2})_{\infty}}{(-a^{2}q^{2\alpha+2};q^{2})_{\infty}}\big[J_{\alpha+1}(aq^{-1},\mu;q^{2})
J_{\alpha}(a,\lambda;q^{2})
-J_{\alpha+1}(aq^{-1},\lambda;q^{2})J_{\alpha}(a,\mu;q^{2})\big] \label{zero}.$$\label{orth}
\end{proposition}
\begin{proof} Using the $q$-integration by parts formula \eqref{int} and
relations  \eqref{rec} and \eqref{recc}, we get
\begin{eqnarray*}
&&\int_{0}^{a}\frac{x(-x^{2} q^{2};q^{2})_{\infty}}{(-x^{2} q^{2\alpha +2};q^{2})_{\infty}}J_{ \alpha}(x,\lambda;q^{2})J_{\alpha}(x,\mu;q^{2})d_{q}x\\
&=& \frac{1-q}{1-q^{2\alpha+2}}\big[ \frac{(-x^{2};q^{2})_{\infty}}{(-x^{2}q^{2\alpha+2};q^{2})_{\infty}}J_{\alpha+1}(xq^{-1},\lambda;q^{2})J_{\alpha}(x,\mu;q^{2})\big]_{0}^{a}
\end{eqnarray*}
$$+\frac{\lambda^{2} q^{2\alpha +2}}{(1-q^{2\alpha+2})^{2}}\int_{0}^{a}\frac{x(-x^{2} q^{2};q^{2})_{\infty}}
{(-x^{2} q^{2\alpha+4};q^{2})_{\infty}}J_{\alpha+1}(x,\lambda;q^{2})J_{\alpha+1}(x,\mu;q^{2})d_{q}x.$$
Then the interchanging of $ \lambda $ and  $ \mu $ in the last
equation, yields a set of two equations, which can be solved easily.
\end{proof}
\begin{corollary}Let
 $ \alpha >-1$ and $a>0$. The zeros of the function $J_{\alpha}(a,\lambda;q^{2})$  are real.
\end{corollary}
\begin{proof} Suppose $\lambda \neq 0$ is a zero of  $\lambda
\rightarrow J_{\alpha}(a,\lambda;q^{2}).$ We have
\begin{center}
$J_{\alpha}(a,\overline{\lambda};q^{2})=\overline{J_{\alpha}(a,\lambda;q^{2})}=0.$
\end{center}
The formula \eqref{zero} with  $\mu=\overline{\lambda}$ yields
$$(\lambda^{2}-\overline{\lambda}^{2})\int_{0}^{a}\frac{x(-x^{2} q^{2};q^{2})_{\infty}}{(-x^{2} q^{2\alpha +4};q^{2})_{\infty}}|J_{\alpha+1}(x,\lambda;q^{2})|^{2}d_{q}x=0.$$
Now $\lambda^{2}=\overline{\lambda}^{2}$ if and only if $\lambda\in \mathbb{R}$ or \ $\lambda \in i\mathbb{R}$, then in all other cases we have
\begin{equation}
\int_{0}^{a}\frac{x(-x^{2} q^{2};q^{2})_{\infty}}{(-x^{2} q^{2\alpha +4};q^{2})_{\infty}}|J_{\alpha+1}(x,\lambda;q^{2})|^{2}d_{q}x=0.
\end{equation}
Using the definition of the $q$-integral we get
$$\int_{0}^{a}\frac{x(-x^{2} q^{2};q^{2})_{\infty}}{(-x^{2} q^{2\alpha +4};q^{2})_{\infty}}|J_{\alpha+1}(x,\lambda;q^{2}|^{2}d_{q}x$$
$$=(1-q)a^{2}\sum\limits_{k=0}^{\infty}\frac{q^{2k}(-a^{2}q^{2k+2};q^{2})_{\infty}}{(-a^{2}q^{2k+2\alpha+4};q^{2})_{\infty}}|J_{\alpha+1}(aq^{k},\lambda;q^{2})|^{2},$$
then
\begin{center}
$ J_{\alpha+1}(aq^{k},\lambda;q^{2})=0, \,\,k\in \mathbb{N}$,
\end{center}
and $J_{\alpha+1}(.,\lambda;q^{2})$ defines an analytic functions on $\mathbb{C}$. Hence, $ J_{\alpha+1}(.,\lambda;q^{2})=0.$
Now if $\lambda=i\mu$, with $\mu\in \mathbb{R}$ we have
$$J_{\alpha}(a,i\mu;q^2)=\sum\limits_{k=0}^{\infty}\frac{q^{2(^{k}_{2})+2k(\alpha+1)}}{(q^{2},q^{2\alpha+2};q^{2})_{k}}(-\frac{1}{a^{2}};q^{2})_{k}(a
\mu )^{2k}$$
for $\alpha>-1$ this expression cannot be zero, which proves the corollary.
\end{proof}
To obtain an expression for the $q$-integral in formula \eqref{zero}
with $ \lambda=\mu, $ we use l'Hopital's rule. The result is
\begin{align*}
&\int_{0}^{a}\frac{x(-x ^{2}q^{2};q^{2})_{\infty}}{(-x^{2} q^{2\alpha +4};q^{2})_{\infty}}\big(J_{\alpha+1}(x,\lambda;q^{2})\big)^{2}d_{q}x=\frac{(q-1)(1-q^{2\alpha+2})(-a^{2};q^{2})_{\infty}}{2\lambda q^{2\alpha+2}(-a^{2}q^{2\alpha+2};q^{2})_{\infty}}\\
&\Big[J_{\alpha}(a,\lambda;q^{2})\big[\frac{\partial J_{\alpha+1}}{\partial \mu}(aq^{-1},\mu;q^{2})\big]_{\mu=\lambda}-J_{\alpha+1}(aq^{-1},\lambda;q^{2})\big[\frac{\partial J_{\alpha}}{\partial \mu}(a,\mu;q^{2})\big]_{\mu=\lambda}\Big].
\end{align*}
This formula simplifies to
\begin{equation}
\int_{0}^{a}\frac{x(-x^{2} q^{2};q^{2})_{\infty}}{(-x^{2} q^{2\alpha +4};q^{2})_{\infty}}(J_{\alpha+1}(x,\lambda;q^{2}))^{2}d_{q}x\label{orthogonal}
\end{equation}
$$=\frac{(1-q)(1-q^{2\alpha+2})}{2\lambda q^{2\alpha+2}}\frac{(-a^{2};q^{2})_{\infty}}{(-a^{2}q^{2\alpha+2};q^{2})_{\infty}}J_{\alpha+1}(aq^{-1},\lambda;q^{2})\big[\frac{\partial J_{\alpha}}{\partial\mu}(a,\mu;q^{2})\big]_{\mu=\lambda},$$
for $\lambda\neq0$ a real zero of $J_{\alpha}(a,\lambda;q^{2}).$
\begin{lemma}
The non-zero real zeros of $\lambda\rightarrow J_{\alpha}(a,\lambda;q^{2}),$ with $\alpha>-1$ are simple zeros.
\end{lemma}
\begin{proof} Let $\lambda$ be a non-zero real zero of
$J_{\alpha}(a,\lambda;q^{2}), $ with $
\alpha>-1$. \ The integral
$$\int_{0}^{a}\frac{x(-x^{2} q^{2};q^{2})_{\infty}}{(-x^{2} q^{2\alpha +4};q^{2})_{\infty}}|J_{\alpha+1}(x,\lambda;q^{2})|^{2}d_{q}x=\int_{0}^{a}\frac{x(-x^{2} q^{2};q^{2})_{\infty}}{(-x^{2} q^{2\alpha +4};q^{2})_{\infty}}(J_{\alpha+1}(x,\lambda;q^{2}))^{2}d_{q}x$$
is strictly positive. If it were zero, this would imply that the big $q$-Bessel function
is identically zero as in the proof of Corollary 3.2 . Hence, \eqref{orthogonal} implies that $$\big[\frac{\partial J_{\alpha}}{\partial\mu}(a,\mu;q^{2})\big]_{\mu=\lambda}\neq0,$$ which proves the lemma.
\end{proof}
Recall that the order $\rho(f)$ of an entire function $f(z),$  see \cite{a,aaa}, is given by
$$\rho(f)=\limsup_{r\rightarrow \infty}\frac{\ln\ln M(r; \ f)}{\ln r}, \ \ \ \ \ \ \ M(r; \ f)=\max_{|z|\leq r}|f(z)|.$$
\begin{lemma}
For $\alpha>-\frac{1}{2}$ and $a\in\mathbb{R}$, the big $q$-Bessel function $\lambda\rightarrow J_{\alpha}(a,\lambda;q^{2})$ has infinitely many zeros.
\end{lemma}
\begin{proof}
We have $$J_{\alpha}(a,\lambda;q^{2})=\sum\limits_{n=0}^{\infty}a_{n}q^{n^{2}}(\lambda a)^{2n},$$
 with
$$a_{n}=\frac{(-1)^{n}q^{(2\alpha+1)n}}{(q^{2},q^{2\alpha+2};q^{2})_{n}}(-\frac{1}{a^{2}};q^{2})_{n}.$$
 By (Theorem 1.2.5,  \cite{aaa}) it suffices to show that \ $\rho(J_{\alpha}(a,\lambda;q^{2}))=0$.\\Since $\alpha>-1/2$, we have
$$\lim\limits _{n\rightarrow\infty}a_{n}=0,$$
then there exist \ $C>0$ \ such that $ |a_{n}|<C, $ and for $
|\lambda|<r,$ we have
\begin{eqnarray*}
M(r,\lambda\rightarrow J_{\alpha}(a,\lambda;q^{2}))
&\leq & C\sum\limits_{n=0}^{\infty}q^{n^{2}}(ar)^{2n},\\
&<&C\sum\limits_{n=-\infty}^{+\infty}q^{n^{2}}(ar)^{2n}.
\end{eqnarray*}
The Jacobi 's triple identity leads
$$
M(r,\lambda\rightarrow J_{\alpha}(a,\lambda;q^{2}))
<C(q^{2},-(ar)^{2}q,-q/(ar)^{2};q^{2})_{\infty}.
$$
Set $r=\frac{q^{-(N+\varepsilon)}}{a},$ for $-\frac{1}{2}\leq\varepsilon<\frac{1}{2}$ and $ N=0,1,2,...$ . Clearly
\begin{eqnarray*}
(-(ar)^{2}q;q^{2})_{\infty}&=&(-q^{-2(N+\varepsilon)+1};q^{2})_{\infty},\\&=&(-q^{1-2\varepsilon}q^{-2N};q^{2})_{\infty}.
\end{eqnarray*}
We have
$$(-(ar)^{2}q;q^{2})_{\infty}=(-q^{1-2\varepsilon};q^{2})_{\infty}(-q^{1-2\varepsilon-2N};q^{2})_{N}$$
and
\begin{eqnarray*}
(-(ar)^{2}q;q^{2})_{\infty}&=&q^{-(N^{2}+2N\varepsilon)}(-q^{2\varepsilon+1};q^{2})_{N}(-q^{1-2\varepsilon};q^{2})_{\infty},\\
&\leq&q^{-(N^{2}+2N\varepsilon)}(-q^{2\varepsilon+1};q^{2})_{\infty}(-q^{1-2\varepsilon};q^{2})_{\infty},
\end{eqnarray*}
also
\begin{eqnarray*}
(-q/(ar)^{2};q^{2})_{\infty}&=&(-q^{2(N+\varepsilon)+1};q^{2})_{\infty},\\
&=&(-1;q^{2})_{\infty}.
\end{eqnarray*}
Hence
\begin{eqnarray*}
&&\frac{\ln\ln M(r,\lambda\rightarrow J_{\alpha}(a,\lambda;q^{2}))}{\ln q^{-(N+\varepsilon)}},\\
&\leq&\frac{\ln\ln Aq^{-N(N+2\varepsilon)}}{\ln q^{-(N+\varepsilon)}},\\
&=&\frac{\ln N(N+2\varepsilon)+\ln\big(\frac{\ln A}{N(N+2\varepsilon)}-\ln q\big)}{-(N+\varepsilon)\ln q},
\end{eqnarray*}
with $$A=C(-q^{1-2\varepsilon},-q^{1+2\varepsilon},-1;q^{2})_{\infty}.$$
Which implies
\begin{center}
$\rho(\lambda\rightarrow J_{\alpha}(a,\lambda;q^{2}))=0$.
\end{center}
\end{proof}
Let $\alpha>-\frac{1}{2}$ we order the positive zeros of $\lambda\rightarrow J_{\alpha}(a,\lambda;q^{2})$ as
\begin{center}
$0<j_{1}^{\alpha}<j_{2}^{\alpha}<j_{3}^{\alpha}<... \ .$
\end{center}
\section{Orthogonality relation and completeness}
The Proposition 3.1 and relation \eqref{orthogonal} with $a=1$
are useful to state the orthogonality relations for the big
$q$-Bessel functions.
\begin{proposition}
Let $\alpha >-\frac{1}{2}$  and
$0<j_{1}^{\alpha}<j_{2}^{\alpha}<j_{3}^{\alpha}<...$ the
positive zeros of the big $q$-Bessel function $J_{\alpha}(1,\lambda;q^{2})$ then
\begin{eqnarray}
\int_{0}^{1}\frac{x(-x^{2} q^{2};q^{2})_{\infty}}{(-x^{2} q^{2\alpha +4};q^{2})_{\infty}}J_{\alpha+1}(x,j_{n}^{\alpha};q^{2})J_{\alpha+1}(x,j_{m}^{\alpha};q^{2})d_{q}x
\end{eqnarray}
\begin{eqnarray*}
&& \\
&=&\frac{(1-q)(1-q^{2\alpha+2})}{2\lambda q^{2\alpha+2}}\frac{(-1;q^{2})_{\infty}}{(-q^{2\alpha+2};q^{2})_{\infty}}J_{\alpha+1}(q^{-1},j_{n}^{\alpha};q^{2})\big[\frac{\partial J_{\alpha}}{\partial\mu}(1,\mu;q^{2})\big]_{\mu=j_{n}^{\alpha}}\delta_{n,m}.
\end{eqnarray*}
\end{proposition}
We consider the inner product giving by
$$<f,g>=\int_{0}^{1}\frac{t(-t^{2}q^{2};q^{2})_{\infty}}{(-t^{2}q^{2\alpha+4};q^{2})_{\infty}}f(t)\overline{g(t)}d_{q}t.$$
Let $$G(\lambda)=\int_{0}^{1}\frac{x(-x^{2}q^{2};q^{2})_{\infty}}{(-x^{2}q^{2\alpha+4};q^{2})_{\infty}}g(x)J_{\alpha+1}(x,\lambda;q^{2})d_{q}x,$$
and $$h(\lambda)=\frac{G(\lambda)}{J_{\alpha}(1,\lambda;q^{2})}.$$
\begin{lemma}
If $ \alpha>-\frac{3}{2} $ and $g(x)\in
L_{q}^{2}((0,1);\frac{x(-x^{2}q^{2};q^{2})_{\infty}}{(-x^{2}q^{2\alpha+4};q^{2})_{\infty}}d_qx)$ then \ $h(\lambda)$ is entire of order $0$.
\end{lemma}
\begin{proof}
We first show that $G(\lambda)$ is entire of order $0$. From the definition
of the $q$-integral we have
\begin{equation}
G(\lambda)=(1-q)\sum\limits_{k=0}^{\infty}\frac{q^{k}(-q^{2k+2};q^{2})_{\infty}}{(-q^{2k+2\alpha+4};q^{2})_{\infty}}g(q^{k})J_{\alpha+1}(q^{k},\lambda;q^{2})q^{k}.
\end{equation}
The series in (4.2) converges uniformly in any disk $|\lambda |\leq
R $. Hence $G(\lambda)$ is entire and we have $$M(r;G)\leq
M(r;\lambda\rightarrow
J_{\alpha+1}(x,\lambda;q^{2}))\int_{0}^{1}\frac{x(-x^{2}q^{2};q^{2})_{\infty}}{(-x^{2}q^{2\alpha+4};q^{2})_{\infty}}|g(x)|d_{q}x.$$
Since \ $\rho(\lambda\rightarrow J_{\alpha+1}(x,\lambda;q^{2}))=0$ \ we have that $\rho(G)=0.$ \\
Both the numerator and the denominator of \ $h(\lambda)$ \ are entire functions of order \ $0$. \ If we write a factor of $G(\lambda)$ and $J_{\alpha}(1,\lambda;q^{2})$ as canonical products, each factor of  $J_{\alpha}(1,\lambda;q^{2})$ divides out with a factor of $G(\lambda)$  by the hypothesis of theorem 4.4  $h(\lambda)$ is thus entire of order $0$.
\end{proof}
\begin{lemma}
The quotient $\frac{J_{\alpha+1}(q^{m},\lambda;q^{2})}{J_{\alpha}(1,\lambda;q^{2})}$ \ is bounded on the  imaginary axis for \ $m\in \mathbb{N}.$
\end{lemma}
\begin{proof}
We will make use of the simple inequalities
\begin{eqnarray*} (-q^{-2m};q^{2})_{n}q^{2nm}&=&\prod\limits_{j=0}^{n-1}(q^{2m}+q^{2j})\\&\leq&\prod\limits_{j=0}^{\infty}(1+q^{2j})\\&=&(-1;q^{2})_{\infty}, \ m\in \mathbb{N},
\end{eqnarray*}
$$1-q^{2n+2\alpha+2}>1-q^{2\alpha+2},$$
and
$$(-1;q^{2})_{n}\geq1.$$
We get for $\lambda = i\mu, \mu$ real,
\begin{eqnarray*}
J_{\alpha+1}(q^{m},i\mu;q^{2})&=&\sum\limits_{n=0}^{\infty}\frac{q^{2(_{2}^{n})+2n(\alpha+2)}}
{(q^{2},q^{2\alpha+4};q^{2})_{n}}
(-q^{-2m};q^{2})_{n}q^{2nm}\mu^{2n},\\ &\leq& \sum\limits_{n=0}^{\infty}\frac{q^{2(_{2}^{n})+2n(\alpha+1)}}{(q^{2},q^{2\alpha+2};q^{2})_{n}}\frac{1-q^{2\alpha+2}}
{1-q^{2\alpha+2+2n}}(-1;q^{2})_{\infty}\mu^{2n}, \\
&<&\sum\limits_{n=0}^{\infty}\frac{q^{2(_{2}^{n})+2n(\alpha+1)}}{(q^{2},q^{2\alpha+2};q^{2})_{n}}
(-1;q^{2})_{\infty}\mu^{2n}.
\end{eqnarray*}
\begin{eqnarray*}
J_{\alpha}(1,i\mu;q^{2})&=&\sum\limits_{n=0}^{\infty}\frac{q^{2(_{2}^{n})+2n(\alpha+1)}}{(q^{2},q^{2\alpha+2};q^{2})_{n}}
(-1;q^{2})_{n}\mu^{2n},\\
&\geq&\sum\limits_{n=0}^{\infty}\frac{q^{2(_{2}^{n})+2n(\alpha+1)}}{(q^{2},q^{2\alpha+2};q^{2})_{n}}\mu^{2n}.
\end{eqnarray*}
Thus we have $$ 0\leq\frac{J_{\alpha+1}(q^{m},i\mu;q^{2})}{J_{\alpha}(1,i\mu;q^{2})}\leq (-1;q^{2})_{\infty}.$$
\end{proof}
\begin{theorem}
Let \ $\alpha >-\frac{3}{2}$ \ and \ $g(x)\in L_{q}^{1}(0,1)$. \ If
$$\int_{0}^{1}\frac{x(-x^{2}q^{2};q^{2})_{\infty}}{(-x^{2}q^{2\alpha+4};q^{2})_{\infty}}g(x)J_{\alpha+1}(x,j_{n}^{\alpha};q^{2})d_{q}x =0, \ n=0,1,2,... \ ,  $$
then \ $ g(q^{m})=0 $ \ for \ $m=0,1,2,...$  .
\end{theorem}
\begin{proof}
Lemma 4.2 implies that $h(i\mu)$  is bounded. Since $h(\lambda)$  is entire
of order $0$, we can apply one of the versions of the Phragm\'{e}n-Lindel\"{o}f
theorem, see \cite{a} and Lemma 4.2 and conclude that $h(\lambda)$ is
bounded in the entire  $\lambda$-plane. Next by Liouville's theorem we conclude
that $ h(\lambda)$ is constant. Say that \ $h(\lambda) = C$. \ We will prove that $C = 0$.
We have $$G(\lambda)=C J_{\alpha}(1,\lambda;q^{2}),$$
\begin{eqnarray*}
G(\lambda)&=&\int_{0}^{1}\frac{a(-a^{2}q^{2};q^{2})_{\infty}}{(-a^{2}q^{2\alpha+4};q^{2})_{\infty}}g(a)J_{\alpha+1}(a,\lambda;q^{2})d_{q}a,\\
&=&(1-q)\sum\limits_{k=0}^{\infty}\frac{q^{2k}(-q^{2k+2};q^{2})_{\infty}}{(-q^{2k+2\alpha+4};q^{2})_{\infty}}g(q^{k})J_{\alpha+1}(q^{k},\lambda;q^{2}),\\
&=&(1-q)\sum\limits_{k=0}^{\infty}\frac{q^{2k}(-q^{2k+2};q^{2})_{\infty}}{(-q^{2k+2\alpha+4};q^{2})_{\infty}}g(q^{k})\sum\limits_{n=0}^{\infty}
\frac{(-1)^{n}q^{2(_{2}^{n})+2n(\alpha+2)}}{(q^{2},q^{2\alpha+4};q^{2})_{n}}(-\frac{1}{q^{2k}};q^{2})_{n}q^{2kn}\lambda^{2n},\\
&=&(1-q)\sum\limits_{n=0}^{\infty}
\frac{(-1)^{n}q^{2(_{2}^{n})+2n(\alpha+2)}}{(q^{2},q^{2\alpha+4};q^{2})_{n}}\sum\limits_{k=0}^{\infty}\frac{q^{2k}
(-q^{2k+2};q^{2})_{\infty}}{(-q^{2k+2\alpha+4};q^{2})_{\infty}}g(q^{k})(-\frac{1}{q^{2k}};q^{2})_{n}q^{2kn}\lambda^{2n},
\end{eqnarray*}
and
\begin{eqnarray*}
J_{\alpha}(1,\lambda;q^{2})=\sum\limits_{n=0}^{\infty}\frac{(-1)^{n}q^{2(_{2}^{n})+2n(\alpha+1)}}{(q^{2},q^{2\alpha+2};q^{2})_{n}}(-1;q^{2})_{n}\lambda^{2n}.
\end{eqnarray*}
It follows that
$$(1-q)\frac{(-1)^{n}q^{2(_{2}^{n})+2n(\alpha+1)}}{(q^{2},q^{2\alpha+4};q^{2})_{n}}q^{2n}
\sum\limits_{k=0}^{\infty}\frac{q^{2k}(-q^{2k+2};q^{2})_{\infty}}{(-q^{2k+2\alpha+4};q^{2})_{\infty}}g(q^{k})(-\frac{1}{q^{2k}};q^{2})_{n}q^{2kn}$$
$$=\frac{(-1)^{n}q^{2(_{2}^{n})+2n(\alpha+1)}}{(q^{2},q^{2\alpha+2};q^{2})_{n}}(-1;q^{2})_{n}C, \ \ \ n=0,1,2,... \ .$$
Dividing out common factors
then we have
$$(1-q)\frac{1-q^{2\alpha+2}}{1-q^{2\alpha+2n+2}} q^{2n}\sum\limits_{k=0}^{\infty}\frac{q^{2k}
(-q^{2k+2};q^{2})_{\infty}}{(-q^{2k+2\alpha+4};q^{2})_{\infty}}g(q^{k})(-\frac{1}{q^{2k}};q^{2})_{n}q^{2kn}$$
$$=(-1;q^{2})_{n}C, \ \ \ n=0,1,2,...\ $$
and letting $n\rightarrow\infty$ gives $$(-1;q^{2})_{\infty}C=0,$$ then $$C = 0.$$ We can
now conclude that $$G(\lambda) =0,$$ or that is
$$\int_{0}^{1}\frac{x(-x^{2}q^{2};q^{2})_{\infty}}{(-x^{2}q^{2\alpha+4};q^{2})_{\infty}}g(x)J_{\alpha+1}(x,j_{n}^{\alpha};q^{2})d_{q}x=0.$$
We complete the proof with a simple argument that gives $g (q^{m}) = 0,
m = 0,1, . . . $.\\
If $$G(\lambda) = 0,$$
then
$$\sum\limits_{k=0}^{\infty}\frac{q^{2k}(-q^{2k+2};q^{2})_{\infty}}{(-q^{2k+2\alpha+2};q^{2})_{\infty}}g(q^{k})(-\frac{1}{q^{2k}};q^{2})_{n}q^{2kn}=0.$$
Letting  $n\rightarrow\infty$  gives  $g(1) = 0$.  Then dividing by  $q^{2n}$ and again letting $n\rightarrow\infty$ gives $g(q) = 0.$  Continuing this process we have  $g (q^{m})=0$ and
the proof of the theorem is complete.
\end{proof}
\section{Fourier big $q$-Bessel series}
Using the orthogonality relation (4.1), we consider the big $q$-Fourier-Bessel
series, $S_{q}^{(\alpha)}[f]$,  associated with a function $f$,
\begin{equation}
S_{q}^{(\alpha)}[f](x)=\sum\limits_{k=1}^{\infty}a_{k}(f)J_{\alpha+1}(x,j_{k}^{\alpha};q^{2}),
\end{equation}
with the coefficients $a_{k}$ given by
\begin{equation}
a_{k}(f)=\frac{1}{\mu_{k}}\int_{0}^{1}\frac{t(-t^{2}q^{2};q^{2})_{\infty}}{(-t^{2}q^{2\alpha+4};q^{2})_{\infty}}f(t)J_{\alpha+1}(t,j_{k}^{\alpha};q^{2})d_{q}t,
\end{equation}
where
\begin{eqnarray*}
\mu_{k}&=&\|J_{\alpha+1}(x,j_{k}^{\alpha};q^{2})\|_{L_{q}^{2}(0,1)}^{2},
\\&=&
\int_{0}^{1}\frac{x(-x^{2}q^{2};q^{2})_{\infty}}{(-x^{2}q^{2\alpha+4};q^{2})_{\infty}}\big[J_{\alpha+1}(x,j_{k}^{\alpha};q^{2})\big]^{2}d_{q}x.
\end{eqnarray*}
\section{Sampling theorem}
The classical Kramer sampling is as follows \cite{e, h}. Let $ K(x;
\lambda)$ be a function, continuous in \ $\lambda $ \ such that, as
a function of  $x$; $K(x; \lambda)\in L^{2}(I )$  for every real
number $\lambda$, where  $I$ is an interval of the real line. Assume
that there exists a sequence of real numbers \ ${\lambda_{n}}$, with
$ n $  belonging to an indexing set contained in \ $\mathbb{Z},$ \
such that ${K(x; \lambda_{n})}$ \ is a complete orthogonal sequence
of functions of  $L^{2}(I).$ Then for any $F$ of the form
$$F(\lambda)=\int_{I}f(x)K(x,\lambda)dx,$$
where \ $F \in L^{2}(I), $ \ we have
\begin{equation}
F(\lambda)=\lim_{N\rightarrow\infty}\sum_{|n|\leq N}F(\lambda_{n})S_{n}(\lambda),\label{s}
\end{equation}
with $$S_{n}(t)=\frac{\int_{I}K(x,\lambda)\overline{K(x,\lambda_{n})}dx}{\int_{I}|K(x,\lambda_{n})|^{2}dx} $$
The series \eqref{s} converges uniformly wherever \ $ ||K(. , \lambda)||_{L^{2}(I)} $ \ is bounded.\\
Now we give a $q$-sampling theorem for the $q$-integral transform of the form
$$F(\lambda)=\int_{0}^{1}\frac{x(-x^{2}q^{2};q^{2})_{\infty}}{(-x^{2}q^{2\alpha+4};q^{2})_{\infty}}f(x)J_{\alpha+1}(x,\lambda;q^{2})d_{q}x, \hspace{1cm} f\in L^{2}_{q}(0,1), \hspace{1cm} \alpha>-\frac{3}{2}.$$
\begin{theorem}
Let $f$ be a function in $L_{q}^{2}((0,1);\frac{x(-x^{2}q^{2};q^{2})_{\infty}}{(-x^{2}q^{2\alpha+4};q^{2})_{\infty}}d_qx)$. Then the $q$-integral transform
\begin{equation}
F(\lambda)=\int_{0}^{1}\frac{x(-x^{2}q^{2};q^{2})_{\infty}}{(-x^{2}q^{2\alpha+4};q^{2})_{\infty}}f(x)J_{\alpha+1}(x,\lambda;q^{2})d_{q}x,\hspace{1cm} \alpha>-\frac{3}{2}.
\end{equation}
has the point-wise convergent sampling expansion
\begin{equation}
F(\lambda)=\sum_{k=1}^{\infty}2F(j_{k}^{\alpha})\frac{j_{k}^{\alpha}J_{\alpha+1}(1,\lambda;q^{2})}{(\lambda^{2}-(j_{k}^{\alpha})^{2})[\frac{\partial J_{\alpha+1}}{\partial\lambda}(1,\lambda;q^{2})]_{\lambda=j_{k}^{\alpha}}}.\label{ss}
\end{equation}
The series \eqref{ss} converges uniformly over any compact subset of $\mathbb{C}$.
\end{theorem}
\begin{proof}
Set \ $K(x,\lambda)=J_{\alpha+1}(x,\lambda;q^{2}),$ \ and \
$j_{k}^{\alpha}$ \ is the k-th positive zero of \
$J_{\alpha+1}(x,\lambda;q^{2}),$ \ and \ $
\{J_{\alpha+1}(x,j_{k}^{\alpha};q^{2})\}_{k=1}^{\infty} $ \ is a
complete orthogonal sequence of function in \
$L_{q}^{2}((0,1);\frac{x(-x^{2}q^{2};q^{2})_{\infty}}{(-x^{2}q^{2\alpha+4};q^{2})_{\infty}}d_qx).$
\ Hence applying Theorem 6.1, we get
\begin{equation}
F(\lambda)=\sum_{k=1}^{\infty}F(j_{k}^{\alpha})\frac{\int_{0}^{1}\frac{x(-x^{2}q^{2};q^{2})_{\infty}}{(-x^{2}q^{2\alpha+4};q^{2})_{\infty}}
J_{\alpha+1}(x,j_{k}^{\alpha};q^{2})J_{\alpha+1}(x,\lambda;q^{2})d_{q}x}{\int_{0}^{1}\frac{x(-x^{2}q^{2};q^{2})_{\infty}}
{(-x^{2}q^{2\alpha+4};q^{2})_{\infty}}|J_{\alpha+1}(x,j_{k}^{\alpha};q^{2})|^{2}d_{q}x}.
\end{equation}
But \ $J_{\alpha+1}(x,\lambda;q^{2})$ \ is analytic on \ $\mathbb{C}$ \, so is bounded on any compact subset of \ $\mathbb{C}$ \ and hence \ $\|J_{\alpha+1}(x,\lambda;q^{2})\|_{2}$ \ is bounded. Substituting  from (3.1) with \ $\mu=j_{k}^{\alpha} $ \ and from (3.4) we obtained (6.3) and the theorem follows.
\end{proof}
\textbf{Example.}
Define a function $f$ on  $[0,1]$ to be
$$f(t)=\left\{\begin{array}{cc}
                 \frac{1}{1-q} & t=1, \\
              0& otherwise .
             \end{array}
\right.$$ Then
$$F(\lambda)=\int_{0}^{1}\frac{x(-x^{2}q^{2};q^{2})_{\infty}}{(-x^{2}q^{2\alpha+4};q^{2})_{\infty}}f(x)J_{\alpha+1}(x,\lambda;q^{2})d_{q}x=\frac{(-q^{2};q^{2})_{\infty}}
{(-q^{2\alpha+4};q^{2})_{\infty}}J_{\alpha+1}(1,\lambda;q^{2}).$$
Thus, applying Theorem 6.1 gives
\begin{equation}
\frac{1}{2}=\sum\limits_{k=1}^{\infty}\frac{j_{k}^{\alpha}
J_{\alpha+1}(1,j_{k}^{\alpha};q^{2})}{(\lambda^{2}-(j_{k}^{\alpha})^{2})[\frac{\partial J_{\alpha+1}}{\partial\lambda}(1,\lambda;q^{2})
]_{\lambda=j_{k}^{\alpha}}},  \hspace{1cm} \lambda\in \mathbb{C}\backslash \{\pm j_{k}^{\alpha}, k\in \mathbb{N^{\ast}}\}.
\end{equation}
We define the Paley-Wiener space related to the big $q$-Bessel by
$$PW_B=\{ \mathcal{F}_B(f);\,\,\ f\in L_{q}^{2}((0,1);\frac{x(-x^{2}q^{2};q^{2})_{\infty}}{(-x^{2}q^{2\alpha+4};q^{2})_{\infty}}d_qx)\},$$
where the finite big $q$-Hankel transform $\mathcal{F}_B (f)$ is defined by
\begin{equation} \mathcal{F}_B (f)(\lambda )=
\int_{0}^{1}\frac{x(-x^{2}q^{2};q^{2})_{\infty}}{(-x^{2}q^{2\alpha+4};q^{2})_{\infty}}f(x)J_{\alpha+1}(x,\lambda;q^{2})d_{q}x,\hspace{1cm} \alpha>-\frac{3}{2}.\label{equation22}
\end{equation}
By a quite similar argument as in the proof of [Theorem 1, \cite{abr}] and Theorem 4.4, we see that the space $PW_B$ equipped with the inner product $$<f,g>_{PW}=\int_{0}^{1}(\mathcal{F}_Bf)(x)\overline{(\mathcal{F}_Bg)(x)}
\frac{x(-x^{2}q^{2};q^{2})_{\infty}}
{(-x^{2}q^{2\alpha+4};q^{2})_{\infty}}d_qx$$
is a Hilbert space and the finite big $q$-Hankel transform \eqref{equation22} becomes a Hilbert space isometry between $L_{q}^{2}((0,1);\frac{x(-x^{2}q^{2};q^{2})_{\infty}}{(-x^{2}q^{2\alpha+4};q^{2})_{\infty}}d_qx)$
and $PW_A$. Therefore, from [Theorem A, \cite{abr}] we deduce that
the big $q$-Bessel function has an associated reproducing kernel.\\
$\mathbf{Acknowledgements}$\\
This research is supported by NPST Program of King Saud University, project number 10-MAT1293-02.
 
\end{document}